\documentclass[12pt]{amsart}

\usepackage{amsmath} %
\usepackage{amssymb} %
\usepackage{amstext}
\usepackage{amsthm}
\usepackage[utf8]{inputenc}
\usepackage{a4wide}
\usepackage{xcolor}

\newcommand{\dimf}{\operatorname{dim}_f}
\newcommand{\attr}{\mathcal A}
\newcommand{\ltwo}{L^2(\Omega)}

\newcommand{\dual}[2]{\left\langle#1,#2\right\rangle}

\newcommand{\dert}{\partial_t}

\newcommand{\rr}{\mathbb{R}}

\newtheorem{lemma}{Lemma}[section]
\newtheorem{theorem}{Theorem}[section]
\newtheorem{definition}{Definition}[section]
\newtheorem{corollary}{Corollary}[section]

\numberwithin{equation}{section}
\begin{document}

\title[]{On finite-dimensional attractor to nonsmooth reaction-diffusion equations}
\thanks{D.P. is supported by the project No. 25-16592S financed by the Grant Agency of the Czech Republic. \\M.Z. was supported by the Swedish Research Council (Vetenskapsrådet) under Grant No. 2022-05046 awarded to Denis Gaidashev.}


\author[D.~Pra\v{z}\'{a}k]{Dalibor Pra\v{z}\'{a}k}
\address{Charles University, Faculty of Mathematics and Physics, Department of Mathematical Analysis, Sokolovsk\'{a}~83, 186~75~Prague~8, Czech~Republic}
\email{prazak@karlin.mff.cuni.cz}

\author[B. Priyasad]{Buddhika Priyasad}
\address{University of Konstanz, Department of Mathematics and Statistics, Universit\"{a}tsstra\ss e 10, 78464 Konstanz, Germany}
\email{priyasad@uni-konstanz.de}

\author[M. Zelina]{Michael Zelina}
\address{University of Uppsala, Department of Mathematics, Ångströmlaboratoriet, Regementsvägen 10, 751 06 Uppsala, Sweden}
\email{michael.zelina@math.uu.se}


\subjclass[2000]{TODO}

\begin{abstract}
We consider a reaction diffusion equation, where the lower-order
nonlinearity allows for a non-smooth part, described by a maximal
monotone relation. We show that the global attractor has finite
fractal dimension, and also obtain explicit upper estimates in 
terms of the data.
\end{abstract}

\maketitle

\section{Introduction}

We consider reaction-diffusion equation
\begin{equation}	\label{eq1}
\dert u - \kappa \Delta u + \sigma_u = f(u)
\end{equation}
for an unknown function $u=u(t,x):[0,T]\times \Omega \to \rr$,
subject to zero Dirichlet boundary condition in a smooth bounded
$\Omega \subset \rr^d$. While the right-hand side nonlinearity
$f=f(u)$ is a regular function, we allow for $\sigma=\sigma_u$
to be multivalued; more specifically, $\sigma$ and $u$ are related 
by a maximal monotone graph. See Section~\ref{sec2} for precise
assumptions.


\par
It is standard to develop the basic theory of weak solutions
for \eqref{eq1}, based on e.g. on the well-known Yosida
approximation. Note that (forward) uniqueness of solutions
is immediate in view of the monotonicity of $\sigma_u$. Furthermore,
under suitable coercivity assumptions, the equation is dissipative
and has a unique compact global attractor $\attr$.

\par
On the other hand, up to our best knowledge, there are no results
(either positive or negative) concerning the finite-dimensionality
of the attractor that would apply to this situation -- see 
the discussion below on related results. Recall that
without the nonsmooth term $\sigma$, one has
\begin{equation}	\label{est1}
\dimf(\attr) \le c \left( 1 + |\Omega| 
		\left( \frac{L}{\kappa} \right)^{d/2}  \right)
\end{equation}
where $L$ is the Lipschitz constant of $f$, $|\Omega|$
and $d$ are the area and dimension of $\Omega$, and
$c$ is some non-dimensional constant. 
\par
For the sake of the following discussion, assume that 
$\kappa = \operatorname{diam}\Omega = 1$,
and $L>>1$ is the only relevant parameter -- see Section~\ref{sec4}
for corresponding non-dimensionalization of \eqref{eq1}. The estimate
\eqref{est1} then reduces to 
\begin{equation}	\label{est2}
	\dimf(\attr) \le c L^{d/2}
\end{equation}
Set $\sigma=0$, and consider the equation in variations $U=du$
\begin{equation}	\label{eq-var}
\dert U =  \Delta U + f'(u)U
\end{equation}
The $N$-trace of the linearized evolution operator can
be estimated by taking a supremum over all orthonormal
families $\{\phi_j\}_{j=1}^N \subset L^2(\Omega)$ 
\begin{align*}
	& \sup \sum_{j=1}^N \dual{\Delta \phi_j}{\phi_j}
		+ \dual{f'(u)\phi_j}{\phi_j}
\\
	&\le \sum_{j=1}^N  \big( - \lambda_j + L \big)
	\le - c N^{2/d+1} + L N
\end{align*}
Here we have used the min-max principle together with
the eigenvalue asymptotics of $-\Delta$, namely  
$\lambda_j \sim j^{2/d}$ (cf.\  \eqref{eq:Assymptotics} in the
Appendix). 
The last term becomes negative for $N> c L^{d/2}$, which leads
to the estimate \eqref{est2}.
\par
On the other hand, assuming that $f(u)=Lu$ in some neighbourhood
of the equilibrium point $u=0$, 
we see that the linearized operator has about $L^{d/2}$ unstable
modes. These arguments (see e.g. \cite{Temam1997} for a rigorous
treatment) show that indeed, if $\sigma=0$, the estimate 
\eqref{est2} is sharp. Yet, the above reasoning is not 
justified in the presence of
nonsmooth $\sigma$. Considering a regular approximation 
$\sigma_\varepsilon$ does not seem to help either, 
as the attractor might not be stable under such limits.
\par
We therefore recourse to the so-called method of trajectories
(see \cite{MP2002}, \cite{FeireislPrazak2010}),
with the resulting estimate of the order
\begin{equation}	\label{est3}
	\dimf(\attr) \le c L^d  \log^2{L}
\end{equation}
see Theorem~\ref{thm:Dimension} below. Note that this is strictly
larger than \eqref{est2}, and the increase seems to be solely
due to the term $\sigma$. More precisely, if $\sigma$ is not
present, our proof again yields \eqref{est2}.
What, then, is the actual value of the attractor dimension
for systems with nonsmooth terms?
Consider a simple case where $d=1$, $\Omega=(0,2\pi)$,
$\sigma = \operatorname{sign}$ and $f(u)=Lu$. It is easy
to see that the function
\begin{equation}	\label{equi1d}
	u(x) = 
\begin{cases}
	\frac{1-\cos(L^{1/2}x)}{L}, \quad &x\in(0,2\pi/L^{1/2})
\\
	0, \quad &\textrm{elsewhere}
\end{cases}
\end{equation}
is a stationary solution, and moreover, about $L^{1/2}$
such distinct equilibria can be placed in $(0,2\pi)$. 
This seems to suggest that in the presence of a nonsmooth term,
the attractor dimension still cannot be smaller then $c L^{d/2}$.
Nonetheless, the gap between \eqref{est2} and \eqref{est3} 
remains open to futher investigation.

\par \medskip
The paper is organized as follows: in Section~\ref{sec2}, we
formulate precise assumptions on the system \eqref{eq1}.
We also state Theorem~\ref{thm:exist_uni} on the existence of unique 
global weak solution. Finally, we formulate our main 
result, Theorem~\ref{thm:Dimension}, providing an explicit estimate 
of the attractor dimension to \eqref{eq1} in terms of the equation 
parameters.
\par
The rest of the paper is devoted to proving the above results.
In Section~\ref{sec3}, the well-posedness theorem is shown,
based on the Yosida approximation.
We also establish the crucial $L^1$-continuous
dependence of $\sigma$ in terms of the solution $u$
(Lemma~\ref{lem:L1-sigma-estimate}). We then obtain higher order
regularity of solutions and existence of a bounded absorbing set.
This ensures existence of a global attractor.
\par
Finally, the explicit estimate of the attractor dimension is
obtained in Section~\ref{sec4}. Turning first \eqref{eq1} into 
non-dimensional form, we use the space of trajectories to
prove the so-called smoothing property of the evolution
operators (Lemma~\ref{thm:SmoothingProperty}).
This in yields the estimate of attractor dimension.
\par
The whole procedure depends on optimal convering numbers
for the appropriate version of the celebrated Aubin-Lions lemma.
This problem, which can be of independent interest, is treated
in Appendix, building on the approach of \cite{FeireislPrazak2010}.


\par\medskip

Let us conclude this opening section with a review of
some recent related results. Surely enough, 
dynamical systems with maximal monotone graphs have
been considered before. To name a few examples, the so-called
Bingham fluids, and fluids with stick-slip boundary conditions
are studied in \cite{BM17}; see also the references therein.
Cahn-Hilliard equation with
(in authors' terminology) irregular terms have been
considered in \cite{GMS09}, \cite{GMS10}, where the existence of
global attractor was shown and moreover, convergence
of solutions to equilibria was established using the
Simon-Lojasiewicz method. We also mention the paper \cite{Lu13},
which considers the 2D Navier-Stokes equations subject
to so-called Tresca boundary condition. The latter is equivalent
to a  monotone stress-strain relation. Using the
method of trajectories, the author constructs a finite-
dimensional exponential attractor.

\par
More general reaction-diffusion systems combined with 
differential inclusions have been studied by \cite{BV99} 
and \cite{Va00}. The authors show that for non-monotone
relations, attractor can be infinite dimensional. They
also formulate abstract sufficient conditions
for its finite-dimensionality. It is not clear, however, if
their positive result applies to the general system 
\eqref{eq1} above.

\par
We also want to mention a paper \cite{EZ08}, where 
a finite-dimensional attractor to porous medium equation,
i.e. essentially problem with a non-smooth but monotone
nonlinearity, was shown using a similar technique based
on testing via the sign function, the Kato inequality 
and $L^1$-estimates.

\section{Problem formulation}	\label{sec2}

Concerning the equation \eqref{eq1}, we impose the 
following standing
assumptions:

\begin{itemize}

\item there exists $\Gamma \subset \rr^2$ a maximal monotone graph,
which corresponds to the multivalued function $\sigma(\cdot)$;
moreover, $\Gamma$ is coercive in the following sense:
\begin{equation}	\label{lower-G}
(u,\sigma) \in \Gamma \implies
\sigma \cdot u \ge c_1|\sigma|^{p} + c_1' |u|^{p'} - c_0
\end{equation}
for some $c_1$, $c_1'$, $c_0>0$ and $p>1$, $p'=p/(p-1)$;

\item $f:\rr \to \rr$ is a $C^1$ function such that 
\begin{equation}	\label{upper-F}
	|f(r)| \le c_2, \qquad f'(r) \le L
\end{equation}
for all $r\in \rr$ with some $c_2$, $L>0$

\end{itemize}


\begin{definition} By weak solution to \eqref{eq1} we understand a couple $(u,\sigma)$ such that
\begin{align}
	u &\in L^2(0,T;W^{1,2}_0(\Omega))	
		\cap L^{p'}((0,T)\times\Omega),
\\
	\sigma & \in L^p((0,T)\times \Omega) \\
	\dert u &\in L^2(0,T;W^{-1,2}(\Omega)) 
			+ L^{p'}((0,T)\times \Omega )
\end{align}
such that \eqref{eq1} holds in the sense of distributions,
and moreover
\begin{equation}
		(u,\sigma) \in \Gamma \quad
	\textrm{a.e.\ in\ } (0,T)\times \Omega
\end{equation}
\end{definition}

\begin{theorem}\label{thm:exist_uni} 
Let $u_0\in L^2(\Omega)$ and $T>0$. Then problem \eqref{eq1} 
admits a unique weak solution $(u,\sigma)$ on $[0,T]$ 
in the sense of Definition~2.1. Moreover,
\[
u\in C([0,T];L^2(\Omega)),
\qquad
u(0)=u_0,
\]
and the solution depends continuously on the initial data $u_0 \in \ltwo$.
\end{theorem}

Before we formulate our main result, we summarize here the standard notions of dynamical systems. For more details we refer the reader to the classical monographs \cite{Temam1997} or \cite{Robinson2001}. We also point out the survey \cite{Zelik2023} and the extensive literature therein.

Let $X$ be a normed space. Family of mappings $\mathcal{S}_t : X \to X$ is referred to as semigroup, provided that $\mathcal{S}_0$ is an identity mapping and $\mathcal{S}_{s+t} = \mathcal{S}_s \mathcal{S}_t$ for all $s$, $t \geq 0$. The couple $(\mathcal{S}_t, X)$ is called a dynamical system if the mapping $(t, x) \mapsto \mathcal{S}_t x$ is continuous. Note that in virtue of Theorem~\ref{thm:exist_uni}, the system \eqref{eq1} generates a dynamical system on the space
$X=L^2(\Omega)$, given by its solution operators $(t,u_0) \mapsto
u(t)$.
\par
Next, a set $\mathcal{A} \subset X$ is referred to as a global attractor to the dynamical system $(\mathcal{S}_t, X)$ provided that the following conditions are met:
\begin{enumerate}
\item[(i)] $\mathcal{A}$ is compact in $X$,
\item[(ii)]  $\mathcal{S}_t \mathcal{A} = \mathcal{A}$ for all $t \geq  0$, and
\item[(iii)]  for any bounded $\mathcal{B} \subset X$ there holds 
\begin{equation*}
\operatorname{dist} (\mathcal{S}_t \mathcal{B}, \mathcal{A}) \rightarrow 0 \text{ as } t \rightarrow \infty, 
\end{equation*}
where $\operatorname{dist} (\mathcal{B}, \mathcal{A})$ is the standard Hausdorff semi-distance of the set $\mathcal{B}$ from the set $\mathcal{A}$, defined as $\operatorname{dist} (\mathcal{B}, \mathcal{A}) = \sup\limits_{a \in \mathcal{A}} \inf\limits_{b\in \mathcal{B}} || b - a||_X$.
\end{enumerate}
Finally, to capture some kind of quantitative information about $\mathcal{A}$, we introduce its fractal dimension. That is
\begin{equation*}
\dim_f^X (\mathcal{A}) = \limsup\limits_{\varepsilon \to 0_+} \frac{\log N_X (\mathcal{A}, \varepsilon)}{- \log \varepsilon} \, ,
\end{equation*}
where $N_X (\mathcal{A}, \varepsilon)$ denotes the minimal number of $\varepsilon$-balls in $X$ needed to cover $\mathcal{A}$.

\begin{theorem}\label{thm:Dimension} The dynamical system generated by \eqref{eq1} has a global attractor $\attr \subset \ltwo$. Moreover, its fractal dimension can be estimated as
\begin{equation} \label{eq:DimensionEstimate}
\dim_f^{L^2(\Omega)}(\attr) 
\leq c 
|\Omega| \bigg( \frac{L}{\kappa} \bigg)^{d}  
\bigg[1 + \log^2  \bigg( |\Omega|^{\frac{1}{d}}\frac{L}{\kappa}  \bigg) \bigg] 
	\,,
\end{equation}
where $c > 0$ is a scale-invariant constant, depending
only on $d$, and the shape of $\Omega$.
\end{theorem}

\section{Wellposedness and global attractor}	\label{sec3}

In this section we first prove Theorem~\ref{thm:exist_uni}.

\begin{proof}[Proof of Theorem~\ref{thm:exist_uni}]
For $\varepsilon>0$, let
\[
J_\varepsilon=(I + \varepsilon\Gamma)^{-1}
\]
be the resolvent of $\Gamma$, and define the Yosida approximation
\[
\sigma_\varepsilon(r)
=
\frac{1}{\varepsilon}\bigl(r-J_\varepsilon r\bigr),
\qquad r\in\mathbb R .
\]
It is well known that $\sigma_\varepsilon$ is single-valued, monotone and globally Lipschitz;
see, for instance, \cite[Chapter~2]{Barbu2010} or \cite[Section~2.4]{Brezis2010}. Moreover,
\[
\bigl(J_\varepsilon r,\sigma_\varepsilon(r)\bigr)\in\Gamma,
\qquad
r=J_\varepsilon r + \varepsilon\sigma_\varepsilon(r).
\]
Hence, by \eqref{lower-G},
\begin{equation}
\label{eq:coercivity-yosida}
\sigma_\varepsilon(r)r
=
\sigma_\varepsilon(r)J_\varepsilon r
+
\varepsilon|\sigma_\varepsilon(r)|^2
\geq
c_1|\sigma_\varepsilon(r)|^p
+
c_1'|J_\varepsilon r|^{p'}
-c_0 .
\end{equation}

We consider the approximate problem
\begin{equation}
\label{eq:approximate-problem}
\begin{cases}
\partial_tu_\varepsilon
-\kappa\Delta u_\varepsilon
+\sigma_\varepsilon(u_\varepsilon)
=
f(u_\varepsilon),
&\text{in }(0,T)\times\Omega,
\\
u_\varepsilon=0,
&\text{on }(0,T)\times\partial\Omega,
\\
u_\varepsilon(0)=u_0,
&\text{in }\Omega .
\end{cases}
\end{equation}
Since $\sigma_\varepsilon$ is globally Lipschitz and monotone, while $f$ is smooth and
bounded, the standard Galerkin method yields a unique weak solution
\[
u_\varepsilon
\in
L^2(0,T;W^{1,2}_0(\Omega))
\cap
C([0,T];L^2(\Omega));
\]
see \cite[Chapter~3]{Temam1997}.

Testing \eqref{eq:approximate-problem} by $u_\varepsilon$ gives
\[
\frac12\frac{d}{dt}\|u_\varepsilon\|_{L^2(\Omega)}^2
+
\kappa\|\nabla u_\varepsilon\|_{L^2(\Omega)}^2
+
\int_\Omega
\sigma_\varepsilon(u_\varepsilon)u_\varepsilon\,dx
=
\int_\Omega
f(u_\varepsilon)u_\varepsilon\,dx .
\]
Using \eqref{eq:coercivity-yosida} and the bound $|f(r)|\leq c_2$, we infer
\begin{align}
\frac12\frac{d}{dt}\|u_\varepsilon\|_{L^2(\Omega)}^2
&+
\kappa\|\nabla u_\varepsilon\|_{L^2(\Omega)}^2
+
c_1\|\sigma_\varepsilon(u_\varepsilon)\|_{L^p(\Omega)}^p
+
c_1'\|J_\varepsilon u_\varepsilon\|_{L^{p'}(\Omega)}^{p'}
\nonumber
\\
&\leq
c_0|\Omega|
+
c_2|\Omega|^{1/2}\|u_\varepsilon\|_{L^2(\Omega)} .
\label{eq:energy-estimate}
\end{align}
Applying Young's inequality and Gronwall's lemma, we obtain
\begin{equation}
\label{eq:uniform-estimates}
\begin{aligned}
&
\|u_\varepsilon\|_{L^\infty(0,T;L^2(\Omega))}
+
\|u_\varepsilon\|_{L^2(0,T;W^{1,2}_0(\Omega))}
\\
&\qquad
+
\|\sigma_\varepsilon(u_\varepsilon)\|_{L^p((0,T)\times\Omega)}
+
\|J_\varepsilon u_\varepsilon\|_{L^{p'}((0,T)\times\Omega)}
\leq C_T ,
\end{aligned}
\end{equation}
where the constant $C_T$ is independent of $\varepsilon$.

From \eqref{eq:approximate-problem} we also have
\[
\partial_tu_\varepsilon
=
\kappa\Delta u_\varepsilon
-
\sigma_\varepsilon(u_\varepsilon)
+
f(u_\varepsilon).
\]
Consequently,
\begin{equation}
\label{eq:time-derivative-bound}
\{
\partial_tu_\varepsilon
\}_{\varepsilon>0}
\quad\text{is bounded in}\quad
L^2(0,T;W^{-1,2}(\Omega))
+
L^p((0,T)\times\Omega).
\end{equation}

By \eqref{eq:uniform-estimates}, \eqref{eq:time-derivative-bound}, and the
Aubin--Lions--Simon compactness theorem
\cite[Corollary~4]{Simon1987}, there exist a subsequence, not relabeled, and functions
\[
u
\in
L^2(0,T;W^{1,2}_0(\Omega))
\cap
L^{p'}((0,T)\times\Omega),
\]
\[
\sigma\in L^p((0,T)\times\Omega),
\]
such that
\begin{align}
u_\varepsilon
&\rightharpoonup u
&&\text{weakly in }L^2(0,T;W^{1,2}_0(\Omega)),
\label{eq:weak-conv-u}
\\
u_\varepsilon
&\overset{\ast}{\rightharpoonup}u
&&\text{weakly-* in }L^\infty(0,T;L^2(\Omega)),
\label{eq:weak-star-conv-u}
\\
u_\varepsilon
&\to u
&&\text{strongly in }L^2((0,T)\times\Omega),
\label{eq:strong-conv-u}
\\
\sigma_\varepsilon(u_\varepsilon)
&\rightharpoonup\sigma
&&\text{weakly in }L^p((0,T)\times\Omega),
\label{eq:weak-conv-eta}
\\
J_\varepsilon u_\varepsilon
&\rightharpoonup u
&&\text{weakly in }L^{p'}((0,T)\times\Omega).
\label{eq:weak-conv-resolvent}
\end{align}
Here, (3.10) follows from
\[
u_\varepsilon-J_\varepsilon u_\varepsilon
=
\varepsilon\sigma_\varepsilon(u_\varepsilon)
\longrightarrow 0
\qquad\text{strongly in }
L^p((0,T)\times\Omega).
\]

Moreover, after passing to a further subsequence if necessary, for every
$\varphi\in L^2(\Omega)$,
\[
\sup_{t\in[0,T]}
\left|
\bigl(u_\varepsilon(t)-u(t),\varphi\bigr)_{L^2(\Omega)}
\right|
\longrightarrow 0.
\]
In particular,
\[
u_\varepsilon(t)\rightharpoonup u(t)
\qquad\text{weakly in }L^2(\Omega)
\]
for every $t\in[0,T]$.
Next, since $f$ is continuous and bounded, the strong convergence
\eqref{eq:strong-conv-u} implies
\[
f(u_\varepsilon)\to f(u)
\qquad
\text{strongly in }
L^q((0,T)\times\Omega)
\]
for every finite $q\geq1$. Therefore, passing to the limit in the weak formulation of
\eqref{eq:approximate-problem}, we obtain
\begin{equation}
\label{eq:limit-equation}
\partial_tu
-\kappa\Delta u
+\sigma
=
f(u)
\qquad
\text{in }\mathcal D'((0,T)\times\Omega).
\end{equation}

The standard compactness argument also gives
\[
u\in C_w([0,T];L^2(\Omega)),
\qquad
u(0)=u_0.
\]
It remains to identify $\sigma$ with an element of $\Gamma(u)$.
For $0<\tau\leq T$, set $Q_\tau=(0,\tau)\times\Omega$. For almost
every $\tau\in(0,T)$, the approximate energy identity reads
\[
\frac12\|u_\varepsilon(\tau)\|_{L^2(\Omega)}^2
+\kappa\int_0^\tau
\|\nabla u_\varepsilon(t)\|_{L^2(\Omega)}^2\,dt
+\int_{Q_\tau}
\sigma_\varepsilon(u_\varepsilon)u_\varepsilon\,dx\,dt
=
\frac12\|u_0\|_{L^2(\Omega)}^2
+\int_{Q_\tau}
f(u_\varepsilon)u_\varepsilon\,dx\,dt.
\]
Since
\[
J_\varepsilon u_\varepsilon
=
u_\varepsilon
-\varepsilon\sigma_\varepsilon(u_\varepsilon),
\]
we have
\[
\int_{Q_\tau}
\sigma_\varepsilon(u_\varepsilon)
J_\varepsilon u_\varepsilon\,dx\,dt
\leq
\int_{Q_\tau}
\sigma_\varepsilon(u_\varepsilon)
u_\varepsilon\,dx\,dt.
\]
The strong convergence of $u_\varepsilon$ and the boundedness and
continuity of $f$ imply
\[
f(u_\varepsilon)u_\varepsilon
\longrightarrow
f(u)u
\qquad\text{strongly in }L^1(Q_\tau).
\]
Therefore, using weak lower semicontinuity in the approximate energy
identity, we obtain
\begin{multline*}
\limsup_{\varepsilon\to0}
\int_{Q_\tau}
\sigma_\varepsilon(u_\varepsilon)
J_\varepsilon u_\varepsilon\,dx\,dt
\leq \\
\frac12\|u_0\|_{L^2(\Omega)}^2
+\int_{Q_\tau}f(u)u\,dx\,dt
-\frac12\|u(\tau)\|_{L^2(\Omega)}^2
-\kappa\int_0^\tau
\|\nabla u(t)\|_{L^2(\Omega)}^2\,dt.    
\end{multline*}
On the other hand, testing the limit equation (3.11) by $u$ and
using the standard chain rule gives
\[
\frac12\|u(\tau)\|_{L^2(\Omega)}^2
+\kappa\int_0^\tau
\|\nabla u(t)\|_{L^2(\Omega)}^2\,dt
+\int_{Q_\tau}\sigma u\,dx\,dt
=
\frac12\|u_0\|_{L^2(\Omega)}^2
+\int_{Q_\tau}f(u)u\,dx\,dt.
\]
Consequently,
\[
\limsup_{\varepsilon\to0}
\int_{Q_\tau}
\sigma_\varepsilon(u_\varepsilon)
J_\varepsilon u_\varepsilon\,dx\,dt
\leq
\int_{Q_\tau}\sigma u\,dx\,dt.
\]
Since
\[
J_\varepsilon u_\varepsilon\rightharpoonup u
\quad\text{in }L^{p'}(Q_\tau),
\qquad
\sigma_\varepsilon(u_\varepsilon)\rightharpoonup\sigma
\quad\text{in }L^p(Q_\tau),
\]
Minty's criterion for the maximal monotone Nemytskii operator induced
by $\Gamma$ yields
\[
(u,\sigma)\in\Gamma
\qquad\text{a.e. in }Q_\tau.
\]
Since this holds for almost every $\tau\in(0,T)$, it follows that
\[
(u,\sigma)\in\Gamma
\qquad\text{a.e. in }(0,T)\times\Omega.
\]
The same chain-rule argument, together with the weak continuity of
$u$, gives
\[
u\in C([0,T];L^2(\Omega)).
\]
Hence $(u,\sigma)$ is a weak solution in the sense of
Definition~2.1.\\

We now prove uniqueness. Let $(u,\sigma_u)$ and $(v,\sigma_v)$ be two weak solutions
corresponding to initial data $u_0$ and $v_0$, respectively. Set
\[
w=u-v .
\]
Subtracting the two equations gives
\[
\partial_tw
-\kappa\Delta w
+
(\sigma_u-\sigma_v)
=
f(u)-f(v).
\]
Testing by $w$ and using the monotonicity of $\Gamma$, we obtain
\[
\frac12\frac{d}{dt}\|w\|_{L^2(\Omega)}^2
+
\kappa\|\nabla w\|_{L^2(\Omega)}^2
\leq
\int_\Omega
(f(u)-f(v))w\,dx .
\]
Since $f'(r)\leq L$, we have
\[
(f(r)-f(s))(r-s)\leq L|r-s|^2,
\qquad
r,s\in\mathbb R .
\]
Therefore,
\[
\frac12\frac{d}{dt}\|w\|_{L^2(\Omega)}^2
\leq
L\|w\|_{L^2(\Omega)}^2 .
\]
An application of Gronwall's lemma yields
\begin{equation}	\label{est:differ}
\|u(t)-v(t)\|_{L^2(\Omega)}^2
\leq
e^{2Lt}
\|u_0-v_0\|_{L^2(\Omega)}^2,
\qquad
t\in[0,T].
\end{equation}
Substituting $u=v$ into the difference equation gives
$\sigma_u=\sigma_v$ in the sense of distributions, and hence almost
everywhere. Thus, the pair $(u,\sigma)$ is unique.
\end{proof}

We have seen that thanks to the equation, $\sigma_u$ is uniquely 
determined by $u$. We now show that more explicitely, its evolution
can be controlled in $L^1$-norm; this will become a key component
in estimating the attractor dimension later on.

\begin{lemma}
\label{lem:L1-sigma-estimate}
Let $(u,\sigma_u)$ and $(v,\sigma_v)$ be two weak solutions of \eqref{eq1} on
$[0,T]$. Then, for every $\ell>0$ such that $2\ell \le T$, 
\begin{equation}
\label{eq:L1-sigma-integrated}
\int_\ell^{2\ell}
\|\sigma_u(t)-\sigma_v(t)\|_{L^1(\Omega)}
\,dt
\leq
\left(
\frac1\ell+2Le^{2L\ell}
\right)
|\Omega|^{1/2}
\int_0^\ell
\|u(t) - v(t)\|_{L^2(\Omega)}
\,dt .
\end{equation}
\end{lemma}

\begin{proof}
We prove the estimate first on the level of the Yosida approximations. Let
$u_\varepsilon$ and $v_\varepsilon$ be the solutions of the approximate problems
corresponding to the initial data of $u$ and $v$, respectively. Set
\[
        w_\varepsilon=u_\varepsilon-v_\varepsilon .
\]
Then
\begin{equation}
\label{eq:approx-difference}
        \partial_tw_\varepsilon
        -\kappa\Delta w_\varepsilon
        +
        \sigma_\varepsilon(u_\varepsilon)
        -
        \sigma_\varepsilon(v_\varepsilon)
        =
        f(u_\varepsilon)-f(v_\varepsilon).
\end{equation}
For $\delta>0$, define
\[
        \phi_\delta(r)
        =
        \frac{r}{\sqrt{\delta^2+r^2}}.
\]
This is a smooth approximation to $\operatorname{sign} r = r/|r|$.
Clearly, $\phi_\delta(w_\varepsilon) \in W^{1,2}_0(\Omega)$ is an
admissible test function for \eqref{eq:approx-difference}, which gives
\begin{align*}
        \frac{d}{dt}
        \int_\Omega
        \sqrt{\delta^2+|w_\varepsilon|^2}\,dx
        &+
        \kappa\delta^2
        \int_\Omega
        \frac{|\nabla w_\varepsilon|^2}
        {(\delta^2+|w_\varepsilon|^2)^{3/2}}
        \,dx
        \\
        &+
        \int_\Omega
        \bigl(
        \sigma_\varepsilon(u_\varepsilon)
        -
        \sigma_\varepsilon(v_\varepsilon)
        \bigr)
        \phi_\delta(w_\varepsilon)\,dx
        \\
        &=
        \int_\Omega
        \bigl(
        f(u_\varepsilon)-f(v_\varepsilon)
        \bigr)
        \phi_\delta(w_\varepsilon)\,dx .
\end{align*}
The diffusion term is nonnegative. Moreover, the assumption
$f'(r)\leq L$ implies
\[
    \bigl(f(r)-f(s)\bigr)\operatorname{sign}(r-s)
    \leq L|r-s|
    \qquad\text{for all }r,s\in\mathbb{R}.
\]
Since $\phi_\delta(r)$ has the same sign as $r$ and
$|\phi_\delta(r)|\leq 1$, it follows that
\[
\int_\Omega
    \bigl(f(u_\varepsilon)-f(v_\varepsilon)\bigr)
    \phi_\delta(w_\varepsilon)\,dx
\leq L\|w_\varepsilon\|_{L^1(\Omega)}.
\]
On the other hand, the monotonicity of $\sigma_\varepsilon$ gives
\[
\bigl(\sigma_\varepsilon(u_\varepsilon)
      -\sigma_\varepsilon(v_\varepsilon)\bigr)
      (u_\varepsilon-v_\varepsilon)\geq 0.
\]
Furthermore, since $\sigma_\varepsilon$ is single-valued,
$u_\varepsilon=v_\varepsilon$ implies
$\sigma_\varepsilon(u_\varepsilon)
 =\sigma_\varepsilon(v_\varepsilon)$. Consequently,
\[
\lim_{\delta\to0^+}
\bigl(\sigma_\varepsilon(u_\varepsilon)
      -\sigma_\varepsilon(v_\varepsilon)\bigr)
      \phi_\delta(w_\varepsilon)
=
\left|
\sigma_\varepsilon(u_\varepsilon)
-\sigma_\varepsilon(v_\varepsilon)
\right|
\quad\text{a.e. in }\Omega.
\]
Therefore, letting $\delta\to0^+$ in the preceding inequality, we obtain,
in the sense of distributions in time,
\begin{equation} \label{eq:L1-diff}
\frac{d}{dt}\|w_\varepsilon(t)\|_{L^1(\Omega)}
+
\|\sigma_\varepsilon(u_\varepsilon(t))
  -\sigma_\varepsilon(v_\varepsilon(t))\|_{L^1(\Omega)}
\leq
L\|w_\varepsilon(t)\|_{L^1(\Omega)}.
\end{equation}
Integrating \eqref{eq:L1-diff} over $(s,2\ell)$, where
$0<s<\ell$, gives
\[
        \int_s^{2\ell}
        \|
        \sigma_\varepsilon(u_\varepsilon(t))
        -
        \sigma_\varepsilon(v_\varepsilon(t))
        \|_{L^1(\Omega)}
        \,dt
        \leq
        \|w_\varepsilon(s)\|_{L^1(\Omega)}
        +
        L
        \int_s^{2\ell}
        \|w_\varepsilon(t)\|_{L^1(\Omega)}
        \,dt .
\]
Restricting the left-hand side to $(\ell,2\ell)$ and using the continuous-dependence estimate established in the proof of Theorem \ref{thm:exist_uni}
\[
        \|w_\varepsilon(t)\|_{L^1(\Omega)}
        \leq
        |\Omega|^{1/2}\|w_\varepsilon(t)\|_{L^2(\Omega)}
        \leq
        e^{L(t-s)}|\Omega|^{1/2}
        \|w_\varepsilon(s)\|_{L^2(\Omega)}
\]
for $s<t<2\ell$, we obtain
\[
        \int_\ell^{2\ell}
        \|
        \sigma_\varepsilon(u_\varepsilon(t))
        -
        \sigma_\varepsilon(v_\varepsilon(t))
        \|_{L^1(\Omega)}
        \,dt
        \leq
        \left(
        1+2L\ell e^{2L\ell}
        \right)
        |\Omega|^{1/2}
        \|w_\varepsilon(s)\|_{L^2(\Omega)} .
\]
Integrating this inequality with respect to $s\in(0,\ell)$ and dividing by $\ell$ gives
\begin{equation}
\label{eq:L1-approx-integrated}
        \int_\ell^{2\ell}
        \|
        \sigma_\varepsilon(u_\varepsilon(t))
        -
        \sigma_\varepsilon(v_\varepsilon(t))
        \|_{L^1(\Omega)}
        \,dt
        \leq
        \left(
        \frac1\ell+2Le^{2L\ell}
        \right)
        |\Omega|^{1/2}
        \int_0^\ell
        \|w_\varepsilon(s)\|_{L^2(\Omega)}
        \,ds .
\end{equation}

It remains to pass to the limit $\varepsilon\to0$. By the compactness obtained in the
existence proof,
\[
        w_\varepsilon\to w
        \qquad
        \text{strongly in }L^2((0,2\ell)\times\Omega),
\]
and
\[
        \sigma_\varepsilon(u_\varepsilon)
        -
        \sigma_\varepsilon(v_\varepsilon)
        \rightharpoonup
        \sigma_u-\sigma_v
        \qquad
        \text{weakly in }L^p((0,2\ell)\times\Omega).
\]
The functional
\[
        \Phi(g)
        =
        \int_\ell^{2\ell}\int_\Omega |g(t,x)|\,dx\,dt
\]
is convex and weakly lower semicontinuous on $L^p((0,2\ell)\times\Omega)$. Therefore,
\[
        \int_\ell^{2\ell}
        \|\sigma_u(t)-\sigma_v(t)\|_{L^1(\Omega)}
        \,dt
        \leq
        \liminf_{\varepsilon\to0}
        \int_\ell^{2\ell}
        \|
        \sigma_\varepsilon(u_\varepsilon(t))
        -
        \sigma_\varepsilon(v_\varepsilon(t))
        \|_{L^1(\Omega)}
        \,dt .
\]
On the other hand, the strong convergence of $w_\varepsilon$ in
$L^2((0,2\ell)\times\Omega)$ implies
\[
        \int_0^\ell
        \|w_\varepsilon(s)\|_{L^2(\Omega)}\,ds
        \to
        \int_0^\ell
        \|w(s)\|_{L^2(\Omega)}\,ds .
\]
Passing to the limit in \eqref{eq:L1-approx-integrated} gives
\eqref{eq:L1-sigma-integrated}.
\end{proof}

Higher regularity of solutions can now be established,
testing the equation with time derivative, and employing
a suitable regularization of the term $\sigma$. This is
the content of the next two lemmata.


\begin{lemma}[Yosida potentials]\label{lem:Yosida-potential}
There exists a proper lower semicontinuous convex function
$\psi:\mathbb{R}\to[0,+\infty)$ such that
\[
\partial\psi=\Gamma.
\]
For $\varepsilon>0$, let
\[
\psi_\varepsilon(r)
:=
\inf_{s\in\mathbb{R}}
\left\{
\psi(s)+\frac{1}{2\varepsilon}|r-s|^2
\right\}.
\]
Then $\psi_\varepsilon\in C^1(\mathbb{R})$,
\[
\psi_\varepsilon'(r)=\sigma_\varepsilon(r),
\quad \psi_\varepsilon(r)
= \psi(J_\varepsilon r)
+\frac{\varepsilon}{2}|\sigma_\varepsilon(r)|^2, \quad 
0\leq\psi_\varepsilon(r)\leq\psi(r),
\]
and there exists a constant $C>0$, independent of
$\varepsilon\in(0,1]$, such that
\begin{equation}
0\leq\psi_\varepsilon(r)
\leq C\bigl(1+|r|^{p'}\bigr),
\qquad r\in\mathbb{R}.
\label{eq:PotentialGrowth}
\end{equation}
\end{lemma}

\begin{proof}
Since $\Gamma\subset\mathbb{R}^2$ is maximal monotone, it is the
subdifferential of a proper lower semicontinuous convex function
$\psi$. We first observe that $\operatorname{dom}\Gamma=\mathbb{R}$.
Indeed, if $(r,\eta)\in\Gamma$, then \eqref{lower-G} and Young's inequality give
\[
c_1|\eta|^p
\leq \eta r+c_0
\leq \frac{c_1}{2}|\eta|^p+C|r|^{p'}+c_0.
\]
Hence,
\[
|\eta|\leq C\bigl(1+|r|^{p'-1}\bigr).
\]
Thus, the values of $\Gamma$ remain bounded when $r$ ranges over a
bounded set. A maximal monotone graph on $\mathbb{R}$ with a finite
endpoint of its domain must be unbounded at that endpoint.
Consequently, $\operatorname{dom}\Gamma=\mathbb{R}$.\\

It follows that $\psi$ is finite and locally absolutely continuous
on $\mathbb{R}$. Assumption \eqref{lower-G} also implies that $\psi$ is
coercive. After subtracting its minimum, we may therefore assume that
$\psi\geq0$. Moreover,
\[
\psi'(r)\in\Gamma(r)
\qquad\text{for almost every }r\in\mathbb{R},
\]
and the preceding estimate gives
\[
|\psi'(r)|
\leq C\bigl(1+|r|^{p'-1}\bigr)
\qquad\text{for almost every }r\in\mathbb{R}.
\]
Integrating this inequality from $0$ to $r$ yields
\[
0\leq\psi(r)\leq C\bigl(1+|r|^{p'}\bigr).
\]
The standard properties of the Moreau--Yosida regularization give
\[
\psi_\varepsilon\in C^1(\mathbb{R}),
\qquad
\psi_\varepsilon'=\sigma_\varepsilon,
\qquad
\psi_\varepsilon(r)
=
\psi(J_\varepsilon r)
+\frac{\varepsilon}{2}|\sigma_\varepsilon(r)|^2,
\]
and $0\leq\psi_\varepsilon\leq\psi$. This proves \eqref{eq:PotentialGrowth}.
\end{proof}


\begin{lemma}[Higher energy estimate]\label{lem:higher-energy}
Let $u_\varepsilon$ be a solution of the approximate problem
\emph{(3.2)}. Then
\begin{equation}
\frac{1}{2}
\|\partial_tu_\varepsilon(t)\|_{L^2(\Omega)}^2
+
\frac{d}{dt}
\left[
\frac{\kappa}{2}
\|\nabla u_\varepsilon(t)\|_{L^2(\Omega)}^2
+
\int_\Omega\psi_\varepsilon(u_\varepsilon(t))\,dx
\right]
\leq
\frac{c_2^2}{2}|\Omega|
\label{eq:TestByDerivative}
\end{equation}
for almost every $t>0$.
\end{lemma}

\begin{proof}
The following computation is first justified for the Galerkin
approximations and then passed to the limit. Testing (3.2) by $\partial_tu_\varepsilon$ and using
$\psi_\varepsilon'=\sigma_\varepsilon$, we obtain
\[
\|\partial_tu_\varepsilon\|_{L^2(\Omega)}^2
+
\frac{d}{dt}
\left[
\frac{\kappa}{2}
\|\nabla u_\varepsilon\|_{L^2(\Omega)}^2
+
\int_\Omega\psi_\varepsilon(u_\varepsilon)\,dx
\right]
=
\int_\Omega
f(u_\varepsilon)\partial_tu_\varepsilon\,dx.
\]
Since $|f(r)|\leq c_2$, Young's inequality gives
\[
\int_\Omega
f(u_\varepsilon)\partial_tu_\varepsilon\,dx
\leq
\frac{1}{2}
\|\partial_tu_\varepsilon\|_{L^2(\Omega)}^2
+
\frac{c_2^2}{2}|\Omega|,
\]
which proves (3.17).
\end{proof}

We are now ready to establish existence of the global attractor 
to \eqref{eq1}.  Recall that $S_t$ is a solution semigroup,
which is well-defined and continuous in $L^2(\Omega)$, in view
Theorem~\ref{thm:exist_uni} and the estimate \eqref{est:differ}.


\begin{lemma}
There exist positively invariant and absorbing sets
$\mathcal{B}_0$, $\mathcal{B}\subset L^2(\Omega)$, 
such that:
\begin{enumerate}
\item $\mathcal{B}$ is bounded and closed in $L^2(\Omega)$, and
\item $\mathcal{B} \subset \mathcal{B}_0$ is closed in $L^2(\Omega)$
and bounded in $W^{1,2}(\Omega)$.
\end{enumerate}
\end{lemma}

\begin{proof}
We work in the normalized setting $\kappa=|\Omega|=1$. All estimates
below are first established for the approximate problem (3.2), with
constants independent of $\varepsilon$, and are then passed to the
limit $\varepsilon\to0$.

Testing (3.2) by $u_\varepsilon$ and using
\[
u_\varepsilon
=
J_\varepsilon u_\varepsilon
+\varepsilon\sigma_\varepsilon(u_\varepsilon),
\]
together with the coercivity assumption \eqref{lower-G}, the bound
$|f|\leq c_2$, Young's inequality and Poincar\'e's inequality, we
obtain constants $\alpha,c,C>0$, independent of $\varepsilon$, such
that
\[
\frac{d}{dt}\|u_\varepsilon\|_{L^2(\Omega)}^2
+\alpha\|u_\varepsilon\|_{L^2(\Omega)}^2
+c\Big(
\|\nabla u_\varepsilon\|_{L^2(\Omega)}^2
+\|\sigma_\varepsilon(u_\varepsilon)\|_{L^p(\Omega)}^p
+\|J_\varepsilon u_\varepsilon\|_{L^{p'}(\Omega)}^{p'}
+\varepsilon
 \|\sigma_\varepsilon(u_\varepsilon)\|_{L^2(\Omega)}^2
\Big)
\leq C.
\]
In particular, Gronwall's inequality gives
\[
\|u_\varepsilon(t)\|_{L^2(\Omega)}^2
\leq
e^{-\alpha t}\|u_\varepsilon(0)\|_{L^2(\Omega)}^2
+
\frac{C}{\alpha}\bigl(1-e^{-\alpha t}\bigr).
\]
The same estimate holds for the limiting weak solution. Therefore,
with
\[
R_0^2:=\frac{2C}{\alpha},
\qquad
\mathcal{B}_0:=
\left\{
u\in L^2(\Omega):
\|u\|_{L^2(\Omega)}\leq R_0
\right\},
\]
the set $\mathcal{B}_0$ is closed, positively invariant and 
absorbing in $L^2(\Omega)$, which proves the first part.

\par
We next obtain a uniform bound at a fixed positive time. Let
$u_\varepsilon(0)\in \mathcal{B}_0$. Integrating the preceding 
differential inequality over $(0,1)$ gives
\begin{equation}
\int_0^1
\Big(
\|\nabla u_\varepsilon(t)\|_{L^2(\Omega)}^2
+\|\sigma_\varepsilon(u_\varepsilon(t))\|_{L^p(\Omega)}^p
+\|J_\varepsilon u_\varepsilon(t)\|_{L^{p'}(\Omega)}^{p'}
+\varepsilon
 \|\sigma_\varepsilon(u_\varepsilon(t))\|_{L^2(\Omega)}^2
\Big)\,dt
\leq C_0
\end{equation}
where $C_0$ is independent of $\varepsilon$ 
and of the initial datum in $\mathcal{B}_0$. Hence, there exists
$\tau_\varepsilon\in(0,1)$ such that
\[
\|\nabla u_\varepsilon(\tau_\varepsilon)\|_{L^2(\Omega)}^2
+
\|J_\varepsilon
  u_\varepsilon(\tau_\varepsilon)\|_{L^{p'}(\Omega)}^{p'}
+
\varepsilon
\|\sigma_\varepsilon
  (u_\varepsilon(\tau_\varepsilon))\|_{L^2(\Omega)}^2
\leq C_0
\]

By the Moreau--Yosida identity and the potential-growth estimate from
Lemma~3.2,
\[
\psi_\varepsilon(r)
=
\psi(J_\varepsilon r)
+\frac{\varepsilon}{2}|\sigma_\varepsilon(r)|^2,
\qquad
\psi(r)\leq C\bigl(1+|r|^{p'}\bigr),
\]
and therefore
\[
\frac{1}{2}
\|\nabla u_\varepsilon(\tau_\varepsilon)\|_{L^2(\Omega)}^2
+
\int_\Omega
\psi_\varepsilon
(u_\varepsilon(\tau_\varepsilon))\,dx
\leq C_{0}.
\]
Integrating the higher-energy estimate from Lemma~3.3 over
$(\tau_\varepsilon,2)$ and using $\psi_\varepsilon\geq0$, we obtain
\[
\|\nabla u_\varepsilon(2)\|_{L^2(\Omega)}
\leq C_{B_0}.
\]
The compactness argument used in the proof of Theorem~2.2 gives
$u_\varepsilon(2)\rightharpoonup u(2)$ in $L^2(\Omega)$.
The uniform $W_0^{1,2}(\Omega)$-bound at time $2$ then yields, after
passing to a subsequence,
\[
u_\varepsilon(2)\rightharpoonup u(2)
\qquad\text{weakly in }W_0^{1,2}(\Omega).
\]
Weak lower semicontinuity therefore gives
\[
\sup_{u_0\in B_0}
\|\nabla S_2u_0\|_{L^2(\Omega)}<\infty.
\]

Define
\[
\mathcal{B}:=\overline{S_2B_0}^{\,L^2(\Omega)}.
\]
Since $S_2\mathcal{B}_0$ is bounded in $W_0^{1,2}(\Omega)$, the
Rellich--Kondrachov theorem shows that $\mathcal{B}$ is compact in
$L^2(\Omega)$. Weak lower semicontinuity also shows that $\mathcal{B}$ is
bounded in $W_0^{1,2}(\Omega)$. Since $\mathcal{B}_0$ is closed
and positively invariant, we also have 
$\mathcal{B}\subset \mathcal{B}_0$.

For every $t\geq0$, the continuity and semigroup properties of $S_t$
give
\[
S_t\mathcal{B}
\subset
\overline{S_tS_2\mathcal{B}_0}^{\,L^2(\Omega)}
=
\overline{S_2S_t\mathcal{B}_0}^{\,L^2(\Omega)}
\subset
\overline{S_2\mathcal{B}_0}^{\,L^2(\Omega)}
=\mathcal{B}\,.
\]
Thus, $\mathcal{B}$ is positively invariant. Finally, every 
bounded subset of $L^2(\Omega)$ eventually enters $\mathcal{B}_0$ 
and, two time units later, it enters $\mathcal{B}$. Hence, 
$\mathcal{B}$ is absorbing.
\end{proof}


In virtue of the above properties of the set $\mathcal B$,
it follows by standard reasoning that 
\begin{equation}
\mathcal{A} := \bigcap_{\tau \geq 0} \overline{ \bigcup_{t \geq \tau} \mathcal{S}_t \mathcal{B}}^{L^2(\Omega)} .\label{eq:ExistenceAttractor}
\end{equation}
is a global attractor, see e.g.
\cite{Temam1997} or \cite{Robinson2001}.

\section{Attractor dimension}	\label{sec4}

The aim of this section is to derive explicit estimates of the fractal dimension of a global attractor $\attr$ in terms of the problem parameters. We begin by observing that \eqref{eq1} can be rescaled so that $\kappa = |\Omega| = 1$, leaving the Lipschitz constant $L$ of the function $f$ as the only relevant parameter. To this end, we introduce new variables $T$, $X$, and $U(T, X)$ through:
\begin{equation*}
t = \tau T , \quad  x = |\Omega|^{\frac{1}{d}} X, \quad \text{and} \quad u(t, x) = \mu U(T, X) ,
\end{equation*}
where $\tau$ and $\mu$ are positive constants. Substituting these expressions into \eqref{eq1}, we find
\begin{equation*}
\frac{\mu}{\tau} \partial_T U(T, X) - \frac{\kappa \mu}{|\Omega|^{\frac{2}{d}}} \Delta_X U(T, X) + \sigma \big( \mu U(T, X) \big) = f \big( \mu U(T, X) \big).
\end{equation*}
Choosing $\mu = |\Omega|^{\frac{2}{d}} / \kappa$ and $\tau = \mu$, and introducing the functions
\begin{equation*}
F(U) = f \bigg( \frac{|\Omega|^{\frac{2}{d}}}{\kappa} U \bigg) , 
\quad \text{and} \quad 
\Sigma (U) = \sigma \bigg( \frac{|\Omega|^{\frac{2}{d}}}{\kappa} U \bigg) ,
\end{equation*}
yields the rescaled equation 
\begin{equation*}
\partial_T U -  \Delta_X U + \Sigma (U) = F (U).
\end{equation*}
Observe that the Lipschitz constant of the new function $F$ is $L |\Omega|^{\frac{2}{d}} / \kappa$.\\

With a slight abuse of notation, we henceforth denote the rescaled
domain, variables, nonlinearities and potential again by
$\Omega$, $t$, $x$, $u$, $f$, $\sigma$ and $\psi$, respectively.
Thus, throughout the remainder of this section, we assume that
$\kappa=|\Omega|=1$. We continue by recalling the energy equality for the difference $w = u - v$, where $u$ and $v$ are solutions to \eqref{eq1} on $[0, T]$ with the corresponding initial data:
\begin{equation}
\frac{1}{2}\, \frac{{\rm d}}{{\rm d}t} || w ||_{L^2(\Omega)}^2 + || \nabla w ||_{L^2(\Omega)}^2 + \int\limits_{\Omega} (\sigma_u - \sigma_v) \cdot w = \int\limits_{\Omega} ( f(u) - f(v)) \cdot w . \label{eq:EnergyDifference}
\end{equation}
In addition, as in the proof of Theorem~\ref{thm:exist_uni}, Grönwall's inequality yields the following continuous dependence estimate:
\begin{equation}
\|w(t)\|_{L^2(\Omega)}^2
\leq
e^{2L(t-s)}
\|w(s)\|_{L^2(\Omega)}^2,
\qquad
0 \leq s \leq t \leq T. \label{eq:MonotonicitySolution}
\end{equation}

We now employ the general method of $\ell$-trajectories introduced in \cite{MP2002}, see also \cite[Section 2.4]{FeireislPrazak2010}. Let $\ell$ be a positive number to be specified below; see Lemma~\ref{thm:SmoothingProperty}. We introduce the set $\mathcal{B}_\ell$ consisting of all $\ell$-trajectories starting from $\mathcal{B}$, that is:
\begin{equation*}
\mathcal{B}_\ell = \big\{ \chi \in L^2(0, \ell; L^2(\Omega)); \chi \text{ is a weak solution to } \eqref{eq1} \text{ on } [0, \ell] \text{ with }  \chi(0) \in \mathcal{B}  \big\} .
\end{equation*}
Next, we introduce two mappings that establish a
one-to-one correspondence between $\mathcal{B}$ and $\mathcal{B}_\ell$. The first mapping $b$ associates with each initial condition $u_0 \in \mathcal{B}$ the unique $\ell$-trajectory starting from $u_0$, i.e. $b : \mathcal{B} \to \mathcal{B}_\ell$ and $b (u_0) = \chi $ with $\chi (0) = u_0$. The second mapping $e$ assigns to each $\ell$-trajectory $\chi $ its endpoint at time 
$\ell$, i.e. $e : \mathcal{B}_\ell \to \mathcal{B}$ and $e(\chi) = \chi (\ell) $. Note that $\chi \in \mathcal{C}_{w}([0, \ell]; L^2(\Omega))$, so the previous equalities are well-defined. It follows from \eqref{eq:MonotonicitySolution} that both mappings $b$ and $e$ are Lipschitz continuous. Therefore, invoking \cite[Lemma 1.2]{MP2002},
\begin{equation}
\operatorname{dim}_f^{L^2(\Omega)} (\mathcal{A})= \operatorname{dim}_f^{L^2(0, \ell; L^2(\Omega))} (\mathcal{A}_\ell) , \label{eq:DimensionLipschitz}
\end{equation}
where $\mathcal{A}_\ell = b(\mathcal{A})$.

The main ingredient in the dimension estimate is the following lemma, which establishes the so-called smoothing property; see \cite[Condition (A6)]{MP2002} or \cite[Definition 2.17]{FeireislPrazak2010}. The key step is to estimate $\partial_t u$ by means of a duality argument in a function space that accommodates the different regularity properties of terms $f(u)$, $\Delta u$, and $\sigma$, the last of which presents the main difficulty.

\begin{lemma}\label{thm:SmoothingProperty}
Let $\ell = 1/L$, and consider $u$ and $v$ to be two weak solutions to \eqref{eq1} on $[0, 2\ell]$ corresponding to the initial conditions $u_0$, $v_0 \in \mathcal{B}$, respectively. Define the space
\begin{equation*}
Y_{\ell}
=  L^\infty(\ell, 2\ell; L^2(\Omega))  \cap L^2(\ell, 2\ell; W^{1,2}(\Omega)) \cap L^\infty(\ell, 2\ell; L^\infty(\Omega)).
\end{equation*}
Setting $w = u - v$, we have:
\begin{align}
|| w ||_{L^2(\ell, 2\ell; W^{1,2}(\Omega))} &\leq c \sqrt{L} || w ||_{L^2(0, \ell; L^2(\Omega))} , \label{eq:SmoothingFirst} \\
|| \partial_t w ||_{Y_{\ell}^*} &\leq c \sqrt{L}  || w ||_{L^2(0, \ell;  L^2(\Omega))} .\label{eq:SmoothingSecond}
\end{align}
\end{lemma}

\begin{proof}
Fix $s \in (0, \ell)$ and integrate \eqref{eq:EnergyDifference} over $(s, 2\ell)$ to find
\begin{equation*}
\int\limits_s^{2\ell} || \nabla w (t)||_{L^2(\Omega)}^2 \, {\rm d}t
\leq  || w (s) ||_{L^2(\Omega)}^2 + \frac{1}{\ell} \int\limits_s^{2\ell} || w (t)||_{L^2(\Omega)}^2 \, {\rm d}t .
\end{equation*}
Using \eqref{eq:MonotonicitySolution}, with $T = 2\ell$ and $\ell = 1/L$, we have
\begin{equation*}
|| w (t) ||_{L^2(\Omega)}^2 \leq e^4 || w (s) ||_{L^2(\Omega)}^2  \quad \text{for any} \quad  0 < s < t < 2\ell .
\end{equation*}
Hence, from the previous inequality we deduce that
\begin{equation*}
\int\limits_{\ell}^{2\ell} || \nabla w (t)||_{L^2(\Omega)}^2 \, {\rm d}t 
\leq (1 + 2 e^4) || w (s) ||_{L^2(\Omega)}^2.
\end{equation*}
Integrating this expression with respect to $s$ over $(0, \ell)$ yields \eqref{eq:SmoothingFirst} with a constant $2e^2$. 

To prove \eqref{eq:SmoothingSecond}, we first recall that the dual of an intersection is naturally identified with the sum of the corresponding dual spaces, and
\begin{equation*}
|| \partial_t w ||_{Y_{\ell}^*}
\leq ||  \partial_t w_1 ||_{L^1(\ell, 2\ell; L^2(\Omega))} +  ||  \partial_t w_2 ||_{L^2(\ell, 2\ell; W^{-1,2}(\Omega))} + ||  \partial_t w_3 ||_{L^1(\ell, 2\ell; L^1(\Omega))}  ,
\end{equation*}
where the decomposition $\partial_t w =  \partial_t w_1 +  \partial_t w_2 +  \partial_t w_3$ is arbitrary. Therefore, by \eqref{eq1}, we obtain
\begin{equation*}
|| \partial_t w ||_{Y_{\ell}^*}
\leq  ||  f(u) - f(v) ||_{L^1(\ell, 2\ell; L^2(\Omega))}  + || \Delta w ||_{L^2(\ell, 2\ell; W^{-1,2}(\Omega))} + ||  \sigma_u - \sigma_v ||_{L^1(\ell, 2\ell; L^1(\Omega))} .
\end{equation*}
We estimate each term on the right-hand side by duality as follows:
\begin{equation*}
|| \partial_t w ||_{Y_{\ell}^*}
\leq
\sup\limits_{\varphi_1} \int\limits_\ell^{2\ell} \langle f(u) - f(v) , \varphi_1 \rangle +
\sup\limits_{\varphi_2} \int\limits_\ell^{2\ell} \langle \Delta w  , \varphi_2 \rangle + 
\sup\limits_{\varphi_3} \int\limits_\ell^{2\ell} \langle  \sigma_u - \sigma_v , \varphi_3 \rangle .
\end{equation*}
Here $\varphi_1 \in L^\infty(\ell, 2\ell; L^2(\Omega))$, $\varphi_2 \in L^2(\ell, 2\ell; W^{1,2}_0(\Omega))$, and $\varphi_3 \in L^\infty(\ell, 2\ell; L^\infty(\Omega))$, where each test function is taken from the unit ball of the corresponding space. The first term is estimated using \eqref{upper-F}, Hölder's inequality, and \eqref{eq:MonotonicitySolution}:
\begin{align*}
\sup\limits_{\varphi_1} \int\limits_\ell^{2\ell} \langle f(u) - f(v) , \varphi_1 \rangle 
&\leq L \int\limits_\ell^{2\ell}  || w (t) ||_{L^2(\Omega)} \, {\rm d}t 
\leq L \bigg( \ell \int\limits_\ell^{2\ell}  || w (t) ||_{L^2(\Omega)}^2  \, {\rm d}t \bigg)^{\frac{1}{2}} \\
&= L \sqrt{\ell} \bigg( \int\limits_0^{\ell}  || w (\tau + \ell) ||_{L^2(\Omega)}^2 \, {\rm d}\tau  \bigg)^{\frac{1}{2}} 
\leq Le^2 \sqrt{\ell} \bigg( \int\limits_0^{\ell}  || w ||_{L^2(\Omega)}^2  \bigg)^{\frac{1}{2}} .
\end{align*}
For the second term, integration by parts, the zero trace of $\varphi_2$, Hölder's inequality, and \eqref{eq:SmoothingFirst} yield
\begin{equation*}
\sup\limits_{\varphi_2} \int\limits_\ell^{2\ell} \langle \Delta w  , \varphi_2 \rangle
\leq \bigg( \int\limits_\ell^{2\ell} || \nabla w ||_{L^2(\Omega)}^2  \bigg)^{\frac{1}{2}}
\leq 2e^2 \sqrt{L} \bigg( \int\limits_0^{\ell} || w ||_{L^2(\Omega)}^2  \bigg)^{\frac{1}{2}} .
\end{equation*}
For the remaining term, we invoke the crucial estimate \eqref{eq:L1-sigma-integrated} and use Hölder's inequality to achieve
\begin{align*}
\sup\limits_{\varphi_3} \int\limits_\ell^{2\ell} \langle  \sigma_u - \sigma_v , \varphi_3 \rangle
&\leq \Big( \frac{1}{\ell} +  2 L e^{2\ell L} \Big)  \int\limits_0^{\ell} ||w||_{L^2(\Omega)} \\
&\leq \Big( \frac{1}{\sqrt{\ell}} +  2 L \sqrt{\ell} e^{2\ell L} \Big)  \bigg( \int\limits_0^{\ell} ||w||_{L^2(\Omega)}^2 \bigg)^{\frac{1}{2}} .
\end{align*}
These estimates, together with $\ell = 1/L$, imply \eqref{eq:SmoothingSecond} with a constant $6e^2$. 
\end{proof}

We now conclude this section by proving the main dimension estimate using the general result from the Appendix.

\begin{proof}[Proof of Theorem~\ref{thm:Dimension}]
By \eqref{eq:ExistenceAttractor}, the global attractor $\mathcal{A}$ exists. In view of \eqref{eq:DimensionLipschitz}, we aim to show 
\begin{equation*}
\operatorname{dim}_f^{L^2(0, \ell; L^2(\Omega))} (b(\mathcal{A})) 
\leq C \Big[1  + L^{\frac{d}{2}} \log^2 L \Big] L^{\frac{d}{2}}
\end{equation*}
for $\kappa = |\Omega| = 1$. The estimate \eqref{eq:DimensionEstimate} then follows since the Lipschitz constant of $f$ scales as described at the beginning of this section.

Since $b$ is Lipschitz,  $b(\mathcal{A})$ is compact, and we can thus find a ball $B(\chi, R)$ centered at some $\chi \in \mathcal{B}_\ell$ with diameter $R > 0$ such that  $b(\mathcal{A}) \subset B(\chi, R)$. Using Lemma~\ref{thm:SmoothingProperty}, 
\begin{equation*}
b(\mathcal{A}) 
= b \Big( e \big( b(\mathcal{A}) \big) \Big)
\subset b (e (\chi) ) + R \mathcal{M} ,
\end{equation*}
where $\mathcal{M}$ is defined in Lemma~\ref{thm:LemmaNormDerivative}. Invoking  Lemma~\ref{thm:LemmaNormDerivative}, with $A$ and $B$ given by \eqref{eq:SmoothingFirst}--\eqref{eq:SmoothingSecond}, and a small parameter $\varepsilon > 0$, to be chosen later, the set on the right-hand side can be covered by $M$ balls of radii $\sqrt{2}$ in the space $L^2(0, \ell; L^2(\Omega))$ satisfying 
\begin{equation*}
\log M \leq C \Big[1 + \frac{1}{L} \varepsilon^{-\frac{2}{1 - 2\varepsilon}} \big( L^{\frac{1}{1-2\varepsilon}} + L + L^{\frac{d + 2\varepsilon}{2 (1-2\varepsilon)}} L^{\frac{1}{1-2\varepsilon}} \big)  \Big] L^{\frac{d}{2}} . 
\end{equation*}

Since $L^{ (d + 2\varepsilon) / (2 (1-2\varepsilon))} = L^{d/2} L^{ \varepsilon (1 + d) / (1-2\varepsilon)}$, we obtain
\begin{equation*}
\log M \leq C \Big[1  + \varepsilon^{-\frac{2}{1 - 2\varepsilon}}  L^{\frac{\varepsilon (3+d)}{1-2\varepsilon}} L^{\frac{d}{2}}  \Big] L^{\frac{d}{2}} . 
\end{equation*}
Now, if $0 < L \leq e^2$, the desired estimate is trivial. Let us assume that $L > e^2$, we choose $\varepsilon = 2/ (3 + d)\log L$ and have
\begin{align*}
\varepsilon^{-\frac{2}{1 - 2\varepsilon}}  L^{\frac{\varepsilon (3+d)}{1-2\varepsilon}}  
&= \exp \Big[-\frac{2}{1 - 2\varepsilon}   \log \varepsilon + \frac{\varepsilon (3+d)}{1-2\varepsilon} \log L  \Big] \\
&= \exp \Big[\frac{2}{1 - 2\varepsilon}   \log \Big(\frac{3+d}{2}\log L \Big) + \frac{2}{1-2\varepsilon}  \Big] \\
&\leq c  \exp \Big[ \Big( 2 + \frac{4\varepsilon}{1 - 2\varepsilon} \Big)  \log \big(2d\log L \big)   \Big] \\
&\leq c \log^2 L \exp \Big[ \frac{8}{(3+d) \log L - 4}  \log \big(2d\log L \big)    \Big]
\\
&\leq c \log^2 L \exp \Big[ \frac{16}{\log L^{2d}}  \log \big(\log L^{2d} \big)    \Big]
\\
&\leq c \log^2 L .
\end{align*}
Here we used the fact that $d \geq 1$ and that $\log (\log x) / \log x$ is bounded for $x > 2$. Hence,
\begin{equation*}
\log M \leq C \Big[1  + L^{\frac{d}{2}} \log^2 L \Big] L^{\frac{d}{2}} .
\end{equation*}

A simple scaling argument yields a covering with balls of radii $R/2$ and the estimate
\begin{equation*}
\log N_{L^2(0, \ell; L^2(\Omega))} \Big( b(\mathcal{A}) , \frac{R}{2} \Big)
\leq C \Big[1  + L^{\frac{d}{2}} \log^2 L \Big] L^{\frac{d}{2}}  \, .
\end{equation*}
Repeating the construction inductively yields
\begin{equation*}
\log N_{L^2(0, \ell; L^2(\Omega))} \Big( b(\mathcal{A}) , \frac{R}{2^n} \Big)
\leq C n \Big[1  + L^{\frac{d}{2}} \log^2 L \Big] L^{\frac{d}{2}} \, ,
\end{equation*}
where $n \in \mathbb{N}$. Finally, using the definition of the fractal dimension,
\begin{equation*}
\operatorname{dim}_f^{L^2(0, \ell; L^2(\Omega))} (b(\mathcal{A})) 
= \limsup\limits_{n \to +\infty} \frac{\log N_{L^2(0, \ell; L^2(\Omega))} \Big( b(\mathcal{A}) , \frac{R}{2^n} \Big)}{- \log \frac{R}{2^n}} \, ,
\end{equation*}
we see that the desired upper bound follows from the previous inequality. 
\end{proof}

\section{Appendix}	\label{sec5}
We introduce here a family of spaces equipped with a convenient Hilbertian structure, following the construction presented in \cite[Section 9.5]{FeireislPrazak2010}. The main advantage of these spaces is that they allow us to avoid estimating the time derivative directly in the rather complicated dual space appearing in Lemma~\ref{thm:SmoothingProperty}.

Let $\{ w_j \}_{ j = 1}^\infty$ be an orthonormal basis of $L^2(\Omega)$ consisting of eigenfunctions of the Dirichlet Laplacian in $\Omega \subset \mathbb{R}^d$. The corresponding eigenvalues $\{ \lambda_j \}_{ j = 1}^\infty$ satisfy
\begin{equation}
\lambda_j \sim  c |\Omega|^{-\frac{2}{d}} j^{\frac{2}{d}} \quad  \quad \text{ as } \quad  j \to +\infty , \label{eq:Assymptotics}
\end{equation}
see, for example, \cite[Corollary 17.5.8]{Hormander2007}. Hence, any function $v \in L^2(\Omega)$ admits the representation $v = \sum\limits_{j = 1}^\infty a_j w_j$, and Parseval's identity yields
\begin{equation*}
|| v ||_{L^2(\Omega)}^2
= \sum\limits_{j = 1}^\infty a_j^2   ,
\quad \text{ where }
a_j
= (v,  w_j)_{L^2(\Omega)} .
\end{equation*}
Furthermore, if $v \in W^{1,2}(\Omega)$, we have $\nabla v = \sum\limits_{j = 1}^\infty a_j \nabla w_j$. Using $|| v ||_{W^{1,2}(\Omega)}^2 = (\nabla v,  \nabla v )_{L^2(\Omega)}$ and $|| \nabla w_j ||_{W^{1,2}(\Omega)}^2 = \lambda_j || w_j ||_{L^2(\Omega)}^2 = \lambda_j$, we obtain 
\begin{equation*}
|| v ||_{W^{1,2}(\Omega)}^2
= \sum\limits_{j = 1}^\infty a_j^2  \lambda_j    ,
\quad \text{ where }
a_j
= (v,  w_j)_{L^2(\Omega)} .
\end{equation*}
In view of these facts, we introduce the space $H^\beta$, for $\beta \in \mathbb{R}$, of the functions $v$ defined on $\Omega$ by the norm
\begin{equation*}
|| v ||_{H^\beta}^2
= \sum\limits_{j = 1}^\infty a_j^2  \lambda_j^\beta , \quad \text{ where }
a_j
= (v,  w_j)_{L^2(\Omega)} .
\end{equation*}
Observe that $H^{-\beta}$ is the continuous dual of $H^\beta$, with duality induced by the generalized inner product in $L^2(\Omega)$. Moreover, $H^0$ and $H^1$ coincide with $L^2(\Omega)$ and $W_0^{1,2}(\Omega)$, respectively. In addition, from \cite[Section 3]{FeHaRo2022} it follows that
\begin{equation}
H^\beta \hookrightarrow W^{\beta, 2} (\Omega) \cap W^{1,2}_0 (\Omega), \quad \beta \geq 0. \label{eq:EmbeddingSpectralSpace}
\end{equation}

Finally, we note that for a given $d \in \mathbb{N}$, we have $|| v ||_{L^{\infty}(\Omega)} \leq c || v ||_{H^\beta}$ for every $\beta > d/2$, which follows from the embedding $W^{\beta,2}(\Omega) \hookrightarrow L^{\infty}(\Omega)$. Later, we will also need information on the dependence of the embedding constant as $\beta$ approaches $d/2$. More precisely, we claim that
\begin{equation}
|| v ||_{L^{\infty}(\Omega)} \leq \frac{c}{\sqrt{\varepsilon}} || v ||_{H^{d/2 + \varepsilon}} , \label{eq:EmbeddingBehaviour} 
\end{equation}
where $c$ is independent of $\varepsilon > 0$.  To see this, let us consider $\beta = d/2 + \varepsilon$ for a given $\varepsilon > 0$. For $v \in H^\beta$, we have
\begin{equation*}
|v(x)|
= \Big| \sum_{j = 1}^\infty a_j w_j (x) \Big |
\leq \sum_{j = 1}^\infty a_j |w_j (x)| \lambda_j^{\frac{\beta}{2}} \lambda_j^{- \frac{\beta}{2}} 
\leq 
c \Big( \sum_{j = 1}^\infty \lambda_j^{-\beta}  | w_j (x)|^2  \Big)^{\frac{1}{2}} 
\Big( \sum_{j = 1}^\infty a_j^2 \lambda_j^{\beta} \Big)^{\frac{1}{2}} .
\end{equation*}
The second factor is precisely the $H^\beta$-norm of $v$, so it remains to estimate the first factor. For this purpose, we recall the local Weyl's Law, see \cite[Theorem 17.5.3]{Hormander2007}, which states that 
\begin{equation*}
\sum_{ \lambda_j \leq \lambda }  | w_j (x)|^2 \leq c \lambda^{\frac{d}{2}}
\end{equation*}
for some $c > 0$ independent of $x \in \Omega$. We rewrite our sum using dyadic blocks and then use \eqref{eq:Assymptotics} together with the above estimate as follows:
\begin{equation*}
\sum_{j = 1}^\infty \lambda_j^{-\beta}  | w_j (x)|^2 
= \sum_{n = 0}^\infty \sum_{ 2^n \leq \lambda_j <  2^{n+1}} \lambda_j^{-\beta}  | w_j (x)|^2 
\leq c\sum_{n = 0}^\infty 2^{ - \beta n} 2^{\frac{1}{2} nd}
= \frac{c}{1 - 2^{-\varepsilon}} \, .
\end{equation*}
Since $2^{\varepsilon}/(2^{\varepsilon} - 1 ) \sim 1/\varepsilon \log 2$ as $\varepsilon \to 0_+$, we conclude that \eqref{eq:EmbeddingBehaviour} holds.

Similarly, let $\{ \varphi_k \}_{ k = 0}^\infty$ be an orthonormal basis of $L^2(0, \ell)$ given by
\begin{equation*}
\varphi_0 \equiv \frac{1}{\sqrt{\ell}} \, , \quad \varphi_k = \sqrt{\frac{2}{\ell}} \cos \frac{k \pi t}{\ell} \quad \text{ for } k \geq 1 .
\end{equation*}
We set 
\begin{equation}
\mu_0 = \frac{1}{\ell^2} \, , \quad \mu_k = \frac{k^2 \pi^2}{\ell^2} \quad \text{ for } \quad k \geq 1 .\label{eq:Assymptotics2}
\end{equation}
For $\alpha \in \mathbb{R}$, we introduce the space $I^\alpha$ of the functions $v$ defined on $(0, \ell)$, endowed with the norm
\begin{equation*}
|| v ||_{I^\alpha}^2
= \sum\limits_{k = 0}^\infty a_k^2  \mu_k^\alpha , \quad \text{ where }
a_k
= (v, \varphi_k)_{L^2(0, \ell)} .
\end{equation*}
As before, we have 
\begin{equation*}
|| v ||_{I^0}^2 = || v ||_{L^2(0, \ell)}^2 
\quad \text{ and } \quad 
|| v ||_{I^1}^2 = || \partial_t v ||_{L^2(0, \ell)}^2 .
\end{equation*}
In particular, estimate \eqref{eq:EmbeddingBehaviour} also holds for the space $I^\alpha$ with $\alpha = 1/2 + \varepsilon$, This follows from the interpolation between $I^0$ and $I^1$, together with the embedding $W^{\alpha,2}(0, \ell) \hookrightarrow L^{\infty}(0, \ell)$. 

In addition, it will be convenient to introduce the following seminorm $\dot{I}^\alpha$ by excluding the first term in the sum:
\begin{equation*}
|| v ||_{\dot{I}^\alpha}^2
= \sum\limits_{k = 1}^\infty a_k^2  \mu_k^\alpha .
\end{equation*}
We will also use the space $I_0^\alpha$, defined analogously to $I^\alpha$, with $\varphi_k$ replaced by
\begin{equation*}
\psi_k (t) = \sqrt{\frac{2}{\ell}} \sin \frac{k \pi t}{\ell} \quad \text{ for } k \geq 0 .
\end{equation*}

We now combine the above constructions to define a convenient norm for the functions $v$ defined on $(0, \ell) \times \Omega$. We set
\begin{equation*}
|| v ||_{I^\alpha (H^\beta)}^2 
= \sum\limits_{j = 1}^\infty \sum\limits_{k = 0}^\infty a_{j, k}^2  \mu_k^\alpha \lambda_j^\beta , \quad \text{ where }
a_{j, k}
= (v, w_j \varphi_k)_{L^2((0, \ell) \times \Omega)} .
\end{equation*}
Analogously, we introduce the seminorm $\dot{I}^\alpha (H^\beta)$ and the space $I_0^\alpha (H^\beta)$.

We now prove a minor extension of the covering result from \cite[Section 9.5]{FeireislPrazak2010} concerning the components of a space $X$ that will be introduced later. 

\begin{lemma}\label{thm:Covering}
Let $\alpha$, $\beta \geq 0$, $A$, $B > 0$, and
\begin{equation*}
\mathcal{M}_{\alpha, \beta}
= \{ u \in I^0(H^0); || u ||_{I^0(H^1)} \leq A  ,  || u ||_{\dot{I}^{\alpha}(H^{-\beta})} \leq B \} .
\end{equation*}
Then $\mathcal{M}_{\alpha, \beta}$ can be covered by $M_{\alpha, \beta}$ balls of radii $\sqrt{2}$ in $ I^0(H^0)$ such that
\begin{equation}
\log M_{\alpha, \beta} \leq c \big(1 + \ell A^{\frac{\beta}{\alpha}} B^{\frac{1}{\alpha}}   \big)|\Omega| A^d    . \label{eq:CoveringEstimate}
\end{equation}
\end{lemma}

\begin{proof}
In view of the previous notation, we have
\begin{equation*}
\sum\limits_{j = 1}^\infty \sum\limits_{k = 0}^\infty a_{j, k}^2 \lambda_j \leq  A^2 
\quad  \text{and} \quad 
\sum\limits_{j, k = 1}^\infty a_{j, k}^2 \lambda_j^{-\beta} \mu_k^{\alpha}  \leq  B^2  .
\end{equation*}
Summing the above inequalities, we obtain the ellipsoid 
\begin{equation*}
\sum\limits_{j, k = 1}^\infty \frac{a_{j, k}^2}{\eta_{j, k}^2 } \leq 1 
\end{equation*}
in the space $I^0(H^0)$ with the semiaxes $\eta_{j, k}$ defined as
\begin{equation*}
\frac{1}{\eta_{j, 0}^2}  = \frac{\lambda_j}{2A^2}  \, ,
\quad \text{and}  \quad 
\frac{1}{\eta_{j, k}^2} = \frac{\lambda_j}{2A^2}  + \frac{\mu_k^{\alpha}}{2B^2 \lambda_j^{\beta}} \, \text{for } k \geq 1  .
\end{equation*}
By \eqref{eq:Assymptotics} and \eqref{eq:Assymptotics2}, we may relabel $\eta_{j, k}$ as a non-decreasing single-indexed sequence $\eta_n$. By \cite[Lemma 2.35]{FeireislPrazak2010}, we find that $M_{\alpha, \beta}$, the number of balls with radii $\sqrt{2}$ to cover the set $\mathcal{M}_{\alpha, \beta}$, satisfies the upper bound
\begin{equation*}
M_{\alpha, \beta}
\leq 3^{N} \prod_{n = 1}^N \eta_n , 
\end{equation*}
where $N$ is the largest natural number such that $\eta_n$ is at least 1.

We fix $\Lambda \geq 1$. Our goal is to estimate the number of pairs $(j,k)$ satisfying $\eta_{j, k} \geq  \Lambda$, i.e. $1/\eta_{j, k}^2 \leq  1/\Lambda^2$. For every such pair, we must have $\lambda_j \leq 2A^2 / \Lambda^2$. Using \eqref{eq:Assymptotics}, we find $j \leq J = c |\Omega| A^d / \Lambda^d$. Suppose now that $J < 1$ even for $\Lambda = 1$, which means that there is no such $j \in \mathbb{N}$. In this case $1 < \lambda_j / 2A^2$, which implies 
\begin{equation*}
|| u ||_{I^0(H^0)}
= \sum\limits_{j = 1}^\infty \sum\limits_{k = 0}^\infty  a_{j, k}^2
< \sum\limits_{j = 1}^\infty \sum\limits_{k = 0}^\infty  a_{j, k}^2 \frac{\lambda_j}{2A^2}
\leq \frac{1}{2} \, .
\end{equation*}
Therefore, in this case, $\mathcal{M}_{\alpha, \beta}$ is contained in the unit ball in $I^0(H^0)$, and \eqref{eq:CoveringEstimate} is trivial. Thus, let us assume that $J \geq 1$. For any fixed $j \leq J$, we have $\mu_k^{\alpha} \leq 2B^2 \lambda_J^{\beta} /  \Lambda^2$ for $k \geq 1$. Using \eqref{eq:Assymptotics2} and $\Lambda \geq 1$, we find that
\begin{equation*}
k \leq K 
= 1 + c \ell B^{\frac{1}{\alpha}}  \lambda_J^{\frac{\beta}{2\alpha}} \Lambda^{-\frac{1}{\alpha}}  
\leq 1 + c \ell  A^{\frac{\beta}{\alpha}}  B^{\frac{1}{\alpha}} .
\end{equation*}
Consequently,
\begin{align*}
\# \Big\{ n \in \mathbb{N}; \eta_n \geq \Lambda \Big\}
&= \# \Big\{ (j, k) \in \mathbb{N}^2; \frac{1}{\eta_{j, k}^2} \leq \frac{1}{\Lambda^2} \Big\} 
\leq JK \\
&\leq  \big(1 + c \ell A^{\frac{\beta}{\alpha}} B^{\frac{1}{\alpha}}  \big)|\Omega| A^d  \Lambda^{-d}\, .
\end{align*}
Invoking \cite[Lemma 9.10]{FeireislPrazak2010}, we find that
\begin{equation*}
\log \prod_{n = 1}^N \eta_n \leq  \big(1 + c \ell   A^{\frac{\beta}{\alpha}} B^{\frac{1}{\alpha}} \big)|\Omega| A^d .
\end{equation*}
Combining this estimate with the above bound for $M_{\alpha, \beta}$, we obtain \eqref{eq:CoveringEstimate}.
\end{proof}

The previous lemma enables us to formulate the following covering result in the setting of the celebrated Aubin--Lions lemma.

\begin{corollary}\label{cor:Covering}
Let $p \geq 1$ and $0 \leq \alpha \leq 1$ be such that $W^{1 - \alpha, 2}_0 (0, \ell) \hookrightarrow L^{p} (0, \ell)$. Consider a Banach space $Y$ and  $\beta \geq 0$ such that $W^{1, 2}_0 (\Omega) \cap W^{\beta, 2} (\Omega)  \hookrightarrow Y$. Let $A$, $B > 0$. Then  
\begin{equation*}
\mathcal{M}
= \{ u \in L^2(0, \ell; L^2(\Omega)); || u ||_{L^2(0, \ell; W^{1,2}(\Omega))} \leq A  ,  || \partial_t u ||_{L^{p'}(0, \ell; Y^*)} \leq B \} .
\end{equation*}
can be covered by $M$ balls of radii $\sqrt{2}$ in $  L^2(0, \ell; L^2(\Omega))$ such that
\begin{equation*}
\log M \leq c \big(1 + \ell A^{\frac{\beta}{\alpha}} B^{\frac{1}{\alpha}}   \big)|\Omega| A^d    . 
\end{equation*}
\end{corollary}

\begin{proof}
The result will follow from \eqref{eq:CoveringEstimate} with $cB$ instead of $B$ if we show 
\begin{equation*}
|| u ||_{\dot{I}^{\alpha} (H^{-\beta})} 
\leq c || \partial_t u ||_{L^{p'} (0, \ell; Y^*)} .
\end{equation*}
Let $\alpha$, $\beta \in \mathbb{R}$. By the dual representation of the norm
\begin{equation*}
|| u ||_{\dot{I}^\alpha (H^{-\beta})}
= \sup\limits_{\varphi } \int\limits_0^\ell \int\limits_{\Omega} u \varphi \, {\rm d}x \, {\rm d}t , 
\end{equation*}
where the supremum is taken over the unit ball in $\dot{I}^{-\alpha} (H^{\beta})$. Since $\partial_t \psi_k (t) = \mu_k^{1/2} \varphi_k (t)$ for $k \geq 1$, we write
\begin{equation*}
\varphi (t, x)
= \sum\limits_{j, k = 1}^\infty a_{j, k}  w_j (x) \varphi_k (t)
= \sum\limits_{j, k = 1}^\infty \frac{a_{j, k}}{\sqrt{\mu_k}} w_j (x) \partial_t \psi_k (t)
= \partial_t \chi (t, x) ,
\end{equation*}
where 
\begin{equation*}
\chi (t, x) = \sum\limits_{j, k = 1}^\infty \frac{a_{j, k}}{\sqrt{\mu_k}} w_j (x) \psi_k (t) .
\end{equation*}
Using $\psi_0 \equiv 0$ together with the orthonormality of the bases, the following equalities hold:
\begin{align*}
|| \chi ||_{I_0^{1-\alpha} (H^{\beta})}
&= \sum\limits_{j = 1}^\infty \sum\limits_{k = 0}^\infty \bigg( \int\limits_0^\ell \int\limits_{\Omega} \chi (t, x) \psi_k(t) w_j(x) \, {\rm d}x \, {\rm d}t  \bigg)^2  \mu_k^{1-\alpha} \lambda_j^\beta \\
&= \sum\limits_{j, k = 1}^\infty \frac{a_{j, k}^2}{\mu_k} \bigg( \int\limits_0^\ell \psi_k^2(t)  \, {\rm d}t \int\limits_{\Omega} w_j^2(x) \, {\rm d}x   \bigg)^2  \mu_k^{1-\alpha} \lambda_j^\beta  
= \sum\limits_{j, k = 1}^\infty a_{j, k}^2 \lambda_j^\beta \mu_k^{-\alpha} , \\
|| \varphi ||_{\dot{I}^{-\alpha} (H^{\beta})} 
&= \sum\limits_{j, k = 1}^\infty \bigg( \int\limits_0^\ell \int\limits_{\Omega} \varphi (t, x) \varphi_k(t) w_j(x) \, {\rm d}x \, {\rm d}t  \bigg)^2   \lambda_j^\beta \mu_k^{-\alpha}
= \sum\limits_{j, k = 1}^\infty  a_{j, k}^2 \lambda_j^\beta \mu_k^{-\alpha} .
\end{align*}
Therefore, 
\begin{equation*}
|| \chi ||_{I_0^{1-\alpha} (H^{\beta})} 
= || \varphi ||_{\dot{I}^{-\alpha} (H^{\beta})} 
= 1 .
\end{equation*}
Hence, integrating by parts in time, which is possible due to $\chi (0) = \chi (\ell) = 0$, we obtain
\begin{equation}
\int\limits_0^\ell \int\limits_{\Omega} u \varphi \, {\rm d}x \, {\rm d}t 
= \int\limits_0^\ell \int\limits_{\Omega} u \partial_t \chi   \, {\rm d}x \, {\rm d}t 
= - \int\limits_0^\ell \langle  \partial_t u , \chi \rangle \, {\rm d}t . \label{eq:Duality}
\end{equation}

Consider now $\alpha$, $\beta \geq 0$ such that $W^{1 - \alpha, 2}_0 (0, \ell) \hookrightarrow L^p (0, \ell)$ and $W^{1, 2}_0 (\Omega) \cap W^{\beta, 2} (\Omega)  \hookrightarrow Y$. In view of \eqref{eq:EmbeddingSpectralSpace}, we have $I^{1 - \alpha}_0 \hookrightarrow W^{1 - \alpha, 2}_0 (0, \ell)$ and $H^\beta \hookrightarrow W^{1, 2}_0 (\Omega) \cap W^{\beta, 2} (\Omega)$. Therefore, $I_0^{1-\alpha} (H^{\beta}) \hookrightarrow L^p (0, \ell; Y)$, and hence $\chi \in L^p (0, \ell; Y)$. Combining this with \eqref{eq:Duality} yields the desired inequality.
\end{proof}

The following result is more subtle, as it considers the time derivative in the dual of the intersection of several spaces arising in Lemma~\ref{thm:SmoothingProperty}.

\begin{lemma}\label{thm:LemmaNormDerivative}
We have
\begin{equation}
|| u ||_X
\leq \frac{c}{\varepsilon} || \partial_t u ||_{Y^*} , \label{eq:NormDerivative}
\end{equation}
where
\begin{align*}
Y
&=  L^\infty(0, \ell; L^2(\Omega))  \cap L^2(0, \ell; W^{1,2}(\Omega)) \cap L^\infty(0, \ell; L^\infty(\Omega)), \\
X
&=  \dot{I}^\alpha (H^0) + \dot{I}^1 (H^{-1}) + \dot{I}^\alpha (H^{-\beta}) ,
\end{align*}
for $\alpha = 1/2 - \varepsilon$ and $\beta = d/2 + \varepsilon$, where $0 < \varepsilon < 1/2$ is arbitrary.

Moreover, for $A$, $B > 1$, the set
\begin{equation*}
\mathcal{M}
= \{ u \in L^2(0, \ell; L^2(\Omega)); || u ||_{L^2(0, \ell; W^{1,2}(\Omega))}  \leq A  ,  ||  \partial_t u ||_{Y^*} \leq B \} 
\end{equation*}
can be covered by $M$ balls of radii $\sqrt{2}$ in $ I^0(H^0)$ such that
\begin{equation}
\log M \leq C \Big[1 + \frac{\ell}{\varepsilon^{1/\alpha}} \big( B^{\frac{1}{\alpha}} + AB + A^{\frac{\beta}{\alpha}} B^{\frac{1}{\alpha}} \big)  \Big]|\Omega| A^d . \label{eq:CoveringEstimateUnion}
\end{equation}
\end{lemma}

\begin{proof}
By the definition of the norm in $X$,  
\begin{equation*}
|| u ||_X
= \inf \big\{ || u_1 ||_{\dot{I}^\alpha (H^0)} + || u_2 ||_{\dot{I}^1 (H^{-1})} + || u_3 ||_{\dot{I}^\alpha (H^{-\beta})} ; u = u_1 + u_2 + u_3 \big\} ,
\end{equation*}
we consider a decomposition $u = u_1 + u_2 + u_3$ for some $u_1 \in \dot{I}^\alpha (H^0)$, $u_2 \in \dot{I}^1 (H^{-1})$, and $u_3 \in \dot{I}^\alpha (H^{-\beta})$. As the dual of the intersection corresponds to the sum of duals and $\partial_t u = \partial_t u_1 + \partial_t u_2 + \partial_t u_3$, it therefore suffices to establish the following inequalities:
\begin{align*}
|| u_1 ||_{\dot{I}^\alpha (H^0)} 
&\leq \frac{c}{\sqrt{\varepsilon}} || \partial_t u_1 ||_{L^1 (0, \ell; L^{2}(\Omega))} , \\
|| u_2 ||_{\dot{I}^1 (H^{-1})} 
&\leq c || \partial_t u_2 ||_{L^2 (0, \ell; W^{-1,2}(\Omega))}, \\
|| u_3 ||_{\dot{I}^\alpha (H^{-\beta})} 
&\leq \frac{c}{\varepsilon} || \partial_t u_3 ||_{L^1 (0, \ell; L^{1}(\Omega))}.
\end{align*}
We argue as in the proof of Corollary~\ref{cor:Covering} to find \eqref{eq:Duality} for each $u_k$. Now, for $\alpha = 1$ and $\beta = -1$, we obtain the second inequality using $\chi \in I_0^0 (H^1) \hookrightarrow L^2 (0, \ell; W_0^{1,2} (\Omega))$, where the embedding follows from $H^1 = W_0^{1,2} (\Omega)$. Next, let $\alpha = 1/2 - \varepsilon$, $\varepsilon > 0$, and $\beta = 0$. Since $\chi \in I_0^{1 - \alpha} (H^{0}) \hookrightarrow L^\infty (0, \ell; L^2 (\Omega))$, the first inequality follows from \eqref{eq:EmbeddingBehaviour}. Likewise, the third inequality follows from the embedding  $H^{\beta}(\Omega) \hookrightarrow L^{\infty}(\Omega)$, valid for $\beta > d / 2$, together with \eqref{eq:EmbeddingBehaviour}.

It remains to show the covering estimate \eqref{eq:CoveringEstimateUnion}. In view of \eqref{eq:NormDerivative}, it suffices to cover the set
\begin{equation*}
\widetilde{\mathcal{M}}
= \Big\{ u \in I^0(H^0); || u ||_{I^0(H^1)}  \leq A  ,  || u ||_{X} \leq \frac{c}{\varepsilon} B \Big\}  .
\end{equation*}
First of all, any $u \in \widetilde{\mathcal{M}}$ satisfies
\begin{equation*}
\sum\limits_{j = 1}^\infty \sum\limits_{k = 0}^\infty a_{j, k}^2 \lambda_j \leq A^2 , 
\end{equation*}
where $a_{j, k}$ are the projections of $u$ in $I^0(H^0)$; more precisely $a_{j, k}
= (u, w_j \varphi_k)_{L^2((0, \ell) \times \Omega)}$. By the definition of $X$, we consider an arbitrary decomposition $u = u_1 + u_2 + u_3$ such that
\begin{equation*}
\sum\limits_{j, k = 1}^\infty \big( b_{j, k}^2 \mu_k^\alpha + c_{j, k}^2 \mu_k \lambda_j^{-1} + d_{j, k}^2 \mu_k^\alpha \lambda_j^{-\beta}  \big) \leq \frac{c^2}{\varepsilon^2} B^2 , 
\end{equation*}
where $b_{j, k}$, $c_{j, k}$, and $d_{j, k}$ are the respective projections of $u_1$, $u_2$, and $u_3$ in $I^0(H^0)$. Hence, 
\begin{equation*}
\sum\limits_{j, k = 1}^\infty \Big[ \frac{1}{2A^2}a_{j, k}^2 \lambda_j + \frac{\varepsilon^2}{2c^2 B^2} \big( b_{j, k}^2 \mu_k^\alpha + c_{j, k}^2 \mu_k \lambda_j^{-1} + d_{j, k}^2 \mu_k^\alpha \lambda_j^{-\beta}  \big) \Big] \leq 1 , 
\end{equation*}
Since $u = u_1 + u_2 + u_3$, we see that
\begin{equation*}
a_{j, k}^2 = (b_{j, k} + c_{j, k} + d_{j, k})^2
\end{equation*}
holds for any $j \in \mathbb{N}$ and $k \in \mathbb{N}_0$. Fix the indices $j \in \mathbb{N}$ and $k \in \mathbb{N}_0$, and suppose that $a_{j, k} \neq 0$. Then at least one of $b_{j, k}$, $c_{j, k}$, and $d_{j, k}$ is also non-zero. Assume first that $b_{j,k} \neq 0$, then we relax the above inequality by keeping only the following term corresponding to the index pair $(j, k)$:
\begin{equation*}
\min \{ a_{j, k}^2, b_{j, k}^2 \} \Big( \frac{\lambda_j}{A^2} + \frac{\varepsilon^2 \mu_k^\alpha }{2c^2 B^2} \Big) .
\end{equation*}
The remaining cases are treated analogously. Hence, we find that
\begin{equation*}
\widetilde{\mathcal{M}}
\subset \mathcal{M}_{\alpha, 0} 
\cup \mathcal{M}_{1, 1} 
\cup \mathcal{M}_{\alpha, \beta} , 
\end{equation*}
where the sets on the right-hand side are as defined in Lemma~\ref{thm:Covering}. Denoting by $M$ the number of balls needed to cover $\mathcal{M}$, we see that 
\begin{equation*}
M \leq M_{\alpha, 0} + M_{1, 1} + M_{\alpha, \beta} .
\end{equation*}
Applying \eqref{eq:CoveringEstimate} to each of the three sets yields \eqref{eq:CoveringEstimateUnion}.
\end{proof}

\bibliographystyle{amsplain}
\bibliography{References}

\end{document}